\documentclass[journal,twoside,web]{ieeecolor}
\usepackage{generic}
\usepackage{amsmath,amssymb,amsfonts}
\usepackage{bm}
\usepackage{cite}
\usepackage{graphicx}
\usepackage{algorithm,algorithmic}
\usepackage{hyperref}
\hypersetup{hidelinks=true}
\usepackage{textcomp}
\def\BibTeX{{\rm B\kern-.05em{\sc i\kern-.025em b}\kern-.08em
    T\kern-.1667em\lower.7ex\hbox{E}\kern-.125emX}}
\newtheorem{theorem}{Theorem}
\newtheorem{proposition}{Proposition}
\newtheorem{lemma}{Lemma}
\newtheorem{corollary}{Corollary}

\newtheorem{definition}{Definition}
\newtheorem{remark}{Remark}
\newtheorem{example}{Example}

\newcommand{\Ccl}{\mathcal{C}_{\mathrm{cl}}}
\newcommand{\Grad}{\mathcal{G}}
\newcommand{\R}{\mathbb{R}}
\newcommand{\T}{\mathcal{T}}
\newcommand{\Crit}{\operatorname{Crit}}
\newcommand{\inner}[2]{\langle #1,\,#2\rangle}

\graphicspath{{./figures/}}
\DeclareGraphicsExtensions{.pdf,.eps}

\begin{document}

\title{First-Order Optimization as Minimum-Time Control}

\author{Liraz Mudrik, \IEEEmembership{Member, IEEE}, Isaac Kaminer, and Pramod P.~Khargonekar, \IEEEmembership{Life Fellow, IEEE}
\thanks{This work was supported in part by the Office of Naval Research Science of Autonomy Program under Grant No.\ N0001425GI01545 and Consortium for Robotics Unmanned Systems Education and Research at the Naval Postgraduate School.}
\thanks{L.~Mudrik is with the Stephen B.\ Klein Faculty of Aerospace Engineering, Technion---Israel Institute of Technology, Haifa, 3200003, Israel (e-mail: sliraz@technion.ac.il).}
\thanks{I.~Kaminer is with the Department of Mechanical and Aerospace Engineering, Naval Postgraduate School, Monterey, CA, 93943, USA (e-mail: kaminer@nps.edu).}
\thanks{P.~Khargonekar is with the Department of Electrical
Engineering and Computer Science, Henry Samueli School of Engineering,
University of California, Irvine, CA, 92697, USA (e-mail: pramod.khargonekar@uci.edu).}}

\maketitle

\begin{abstract}
We formulate first-order optimization as a minimum-time control problem. The
iterate is the state, the update, a combination of the gradients observed so
far, is the control, and the points with gradient norm at most a prescribed
tolerance form the target set. For a fixed objective and start, the minimum
number of oracle queries needed to reach the target is a value function: it
measures the complexity of the instance, not the worst case over a class. On a strongly convex quadratic, the
conjugate gradient iterates emerge from the discrete Pontryagin conditions, and
the value is a controllability index. Beyond the quadratic, a Hessian-generated
reachable span replaces the controllability matrix, and feasibility becomes a
reachability question: whether a critical point can be reached at all, and in
how many steps, is read from the span. Curvature is a resource: a checkable
condition certifies that an instance needs fewer steps than the controllability
index of its quadratic model at the minimizer, and the gap can grow without
bound with the dimension. The value is thus a benchmark for the intrinsic
difficulty of each instance, against which any first-order method can be
measured.
\end{abstract}

\begin{IEEEkeywords}
Optimization algorithms, optimal control, minimum-time control, oracle complexity, controllability, conjugate gradient.
\end{IEEEkeywords}

\section{Introduction}
\label{sec:intro}

The complexity of first-order optimization is classically studied in the worst case over a function class. The information-based lower bounds of Nemirovski and Yudin~\cite{NemirovskiYudin1983} and the construction in Nesterov's lectures~\cite{Nesterov2018} both reason about a single structural object: any first-order method produces iterates lying in the affine span of the starting point and the gradients observed so far, and one then exhibits a worst-case instance on which no method in this span class can converge quickly; for oblivious methods, whose schedule is fixed in advance, these bounds hold independently of the dimension~\cite{ArjevaniShamir2016}. Performance estimation~\cite{DroriTeboulle2014,TaylorHendrickxGlineur2017} and the optimized gradient method~\cite{KimFessler2016} make the worst case exact by computing a minimax-optimal method over a class.

A parallel line of work recasts optimization algorithms as dynamical systems or feedback loops and certifies their behavior with control-theoretic tools. We organize this body of work, and our own contribution, into three levels.
At the first level, one analyzes a given algorithm: integral quadratic constraints (IQCs) furnish a semidefinite criterion that certifies a convergence rate for a fixed method~\cite{LessardRechtPackard2016,FazlyabRibeiro2018}, dissipativity and Lyapunov arguments recover the same guarantees in a state-space form~\cite{HuLessard2017}, continuous-time and variational models explain acceleration as a property of an underlying flow~\cite{SuBoydCandes2016,WibisonoWilsonJordan2016,MuehlebachJordan2021}, and the algorithm-as-controller viewpoint makes the feedback interpretation explicit~\cite{Hauswirth2024,BhayaKaszkurewicz2006}. At the second level, one synthesizes an algorithm optimal in the worst case over a class: performance estimation casts the worst-case error of a method as a semidefinite program and optimizes the method's coefficients against it~\cite{DroriTeboulle2014,TaylorHendrickxGlineur2017,KimFessler2016}, and IQC-based synthesis designs accelerated methods with certified rates~\cite{SchererEbenbauer2021,GramlichEbenbauerScherer2022}; the object produced is a single method tuned to the hardest instance in the class.

The subject of this paper, which is the third level, is to study the per-instance complexity by a minimum-time criterion. We fix a single objective $f$ and a single starting point $z_0$ and ask: over the full first-order span class, what is the least number of oracle queries needed to drive the gradient to within a prescribed tolerance of zero? The question is instance-specific rather than worst-case and minimum-time rather than asymptotic-rate or Lyapunov, and we formulate it as the value function of a discrete-time minimum-time optimal control problem, in which the iterate is the state, the span coefficients are the control, and the discrete Pontryagin minimum principle~\cite{Naidu2003,Bertsekas2017} characterizes the optimal schedule. A distinct control-theoretic program reformulates the optimization problem itself as an optimal control problem and synthesizes a runnable method from the maximum principle: Zhang et al.~\cite{ZhangWangXuGuo2024} recover gradient descent and Newton's method as special cases, a successor establishes a superlinear rate for the resulting method~\cite{WangXuGuoZhang2026}, and Ross derives gradient, Newton, and coordinate-descent methods from a control Lyapunov function along an optimal-control trajectory~\cite{Ross2019,Ross2023}. Those works design an algorithm; we place a minimum-time cost on the iteration count, so the value function is a per-instance complexity measure read through controllability, and the use of a minimum-time count in place of a fixed-horizon cost is the principal formal difference. On a strongly convex quadratic the conjugate gradient iterates~\cite{HestenesStiefel1952} emerge from our optimality system, rather than gradient descent or Newton's method.

Read this way, the per-instance count is governed by the reachable structure the instance carries rather than by its conditioning, and it admits a computable certificate: the value is at most the dimension of the reachable gradient span, evaluable from first-order information alone, and for a generalized linear model this is the rank of the data, independent of the ambient dimension. The gap to worst-case analysis is thus quantitative: the value reflects the spectral structure the instance presents rather than its conditioning, so on a clustered spectrum it can fall far below the query count a worst-case-optimal method must pay on the same instance. Our contribution is to obtain this beyond the quadratic: for twice-differentiable objectives the value is governed by a reachable gradient span and a discrete accessibility-type condition~\cite{JakubczykSontag1990,AlbertiniSontag1993}, and curvature can lower the per-instance count below the matched controllability index. The quadratic case is the recognizable specialization, where the reachable set is a Krylov subspace, the value equals the controllability index of a single-input system, and conjugate gradient attains it. The resulting three-way correspondence among the first-order span, the Krylov subspace, and the controllable subspace~\cite{Kalman1963,Sontag1998} is well known in numerical linear algebra~\cite{GolubVanLoan2013,Saad2003} and model reduction~\cite{Antoulas2005}, though rarely stated as a single object. The viewpoint connects to instance-optimal and instance-faster first-order methods: for convex quadratics conjugate gradient is instance-optimal for the function value, realized equivalently by an adaptive heavy-ball method with Polyak step sizes~\cite{GoujaudTaylorDieuleveut2024}, and adaptive methods attain instance-faster rates under refined curvature conditions~\cite{LiuFang2023}; those works synthesize runnable adaptive methods measured by convergence rate, while our quadratic correspondence is the control-theoretic reading of the same instance-optimality and our content beyond the quadratic is the reachability characterization.

The minimum-time value function with a target is, by construction, an object that presumes knowledge of the instance and of its target; this is standard in minimum-time control, where the value function and the implementable feedback law are distinct~\cite{BardiCapuzzoDolcetta1997}, and we use it as a benchmark rather than a runnable algorithm. The closest precedents are continuous-time and Hamilton--Jacobi in nature: the eikonal value function of Bardi and Kouhkouh~\cite{BardiKouhkouh2023} reaches the minimizer set in finite time in some cases and is constructed without prior knowledge of the minimizer, and a finite-termination result for Morse polynomials proceeds through sum-of-squares hierarchies rather than first-order iterations~\cite{LeCongTrinh2018}.

The contributions of this paper are as follows. 
We first formulate the per-instance first-order complexity as the value function of a discrete-time minimum-time optimal control problem and derive its optimality system by dynamic programming and by the discrete minimum principle. On a strongly convex quadratic the optimality system reduces to the residual-orthogonality (Galerkin) conditions of the conjugate gradient method, so the conjugate gradient iterates emerge as the only span-class trajectory satisfying the optimality conditions, not merely as one admissible schedule (Thm.~\ref{thm:selection}). The exact-reaching value then equals the controllability index of the associated single-input pair.
Then, along the controllability axis, we characterize the value for twice-differentiable objectives by a reachability condition in which a Hessian-generated reachable span replaces the controllability matrix (Thm.~\ref{thm:controllability}), recovering the quadratic case as a specialization. The value is bounded by the dimension of the reachable structure the instance carries (Cor.~\ref{cor:certificate}), for a generalized linear model the rank of the data (Cor.~\ref{cor:glm}), independent of the ambient dimension.
Curvature is not only an obstacle but a resource: we show it can drive the per-instance count strictly below the matched controllability index (Thm.~\ref{thm:curvature}). On a reachable structure of dimension $d$, a checkable transversality condition on a coplanarity determinant produces a start that reaches its target in $d-1$ steps where the matched quadratic requires $d$, on an open set of starts. The reduction is by exactly one for $d\in\{3,4\}$; in higher dimension it compounds to a floor of order $\sqrt{2d}$ steps, a gap below the controllability index that grows without bound. This is a per-instance phenomenon with no worst-case analogue: the classical first-order lower bounds are saturated by quadratics (Remark~\ref{rem:minimax}), the very instances that carry the largest per-instance count, so a reduction available only away from the quadratic is invisible to the worst-case account.

The remainder of this paper is organized as follows. Section~\ref{sec:formulation} fixes the first-order method class and the minimum-time value. Section~\ref{sec:solving} derives the optimality system by dynamic programming and the discrete Pontryagin principle and states the benchmark procedure it defines. Section~\ref{sec:necessary} specializes it to the strongly convex quadratic, where conjugate gradient emerges, the value is a controllability index, and the relation to worst-case complexity is recorded. Section~\ref{sec:reach} gives the reachability test for twice-differentiable objectives and the curvature shortcut, and Sec.~\ref{sec:conclusion} concludes. Longer proofs are deferred to the appendix in order of appearance; proofs that are immediate are given with their statements.

\section{Problem Formulation}
\label{sec:formulation}

Let $f:\R^n\to\R$ be continuously differentiable with a global minimizer $z^\star$, not assumed unique, and minimum value $f^\star=f(z^\star)$. A first-order method queries an oracle that returns, at any point $z$, the pair $(f(z),\nabla f(z))$; the oracle is exact. An instance is one objective paired with one start, $(f,z_0)$; worst-case analysis ranges over a class of objectives~\cite{NemirovskiYudin1983,Nesterov2018}, and we fix a single instance. The oracle complexity of a method on an instance is the number of queries it makes.

\subsection{The first-order method class}

A method in the full first-order class $\Ccl$ generates iterates of the form
\begin{equation}
\label{eq:fo-class}
z_{k+1}=z_0-\sum_{j=0}^{k}c_{k,j}\,\nabla f(z_j),\qquad c_{k,j}\in\R,
\end{equation}
where the coefficient array $\{c_{k,j}\}$ is the control variable of the minimum-time problem posed next. The class $\Ccl$ is the set of all iterations of the form~\eqref{eq:fo-class}; a coefficient array chosen out to some horizon is a schedule. Writing $g_j=\nabla f(z_j)$ for the observed gradients, every iterate of~\eqref{eq:fo-class} lies in the affine slice
\begin{equation}
z_K\in z_0-S_K(z_0),\qquad S_K(z_0)=\operatorname{span}\{g_0,\dots,g_{K-1}\},
\label{eq:reachspan}
\end{equation}
and we call $S_K(z_0)$ the $K$-step reachable span from $z_0$. Equation~\eqref{eq:reachspan} is the linear-span condition underlying the worst-case information-based lower bounds~\cite{NemirovskiYudin1983,Nesterov2018}. Gradient descent, the heavy-ball and Nesterov methods, and conjugate gradient are particular choices of the coefficient array. The value defined next is a benchmark: its coefficients may depend on the whole instance, with no restriction on their sign or size. A method that can be run chooses each coefficient from the data observed so far.

\subsection{The minimum-time value function}

Fix a tolerance $\varepsilon>0$. The target is the gradient-tolerance set
\begin{equation}
\label{eq:target}
\T_\varepsilon=\{z\in\R^n:\ \|\nabla f(z)\|\le\varepsilon\},
\end{equation}
and for $\varepsilon=0$ it is the critical set $\T_0=\Crit(f)=\{z\in\R^n:\nabla f(z)=0\}$. Of the three natural targets, gradient norm, function gap, and distance to a minimizer, the gradient norm is the only one computable from the oracle alone, and driving it below a tolerance is the standard stopping rule of first-order optimization. The target is intrinsic in this sense: an iterate certifies its own arrival, with no knowledge of $z^\star$ or $f^\star$.

\begin{definition}[Minimum-time value]
\label{def:value}
The minimum-time value of $z_0$ over $\Ccl$ is
\begin{equation}
\label{eq:value}
V(z_0)=\min\Big\{K\in\mathbb{N}:\ \exists\,\{c_{k,j}\}_{0\le j\le k<K}\ \text{with}\ z_K\in\T_\varepsilon\Big\},
\end{equation}
with $V(z_0)=0$ if $z_0\in\T_\varepsilon$ and $V(z_0)=\infty$ if no finite schedule reaches $\T_\varepsilon$. The exact-reaching value $V_0(z_0)$ is the same minimum with $\T_0$ in place of $\T_\varepsilon$.
\end{definition}

The quantity $V(z_0)$ is the per-instance oracle complexity of driving the gradient below $\varepsilon$: the fewest oracle queries achievable by a schedule of the form~\eqref{eq:fo-class} whose coefficients are chosen with full knowledge of the instance. Since every critical point lies in $\T_\varepsilon$, $V(z_0)\le V_0(z_0)$ for every $\varepsilon>0$, with equality for all sufficiently small $\varepsilon$. Whether any finite schedule reaches the target at all is not automatic; deciding it, instance by instance, is the feasibility half of the problem.

\begin{example}[Gradient descent under gradient dominance]
\label{ex:pl}
Any runnable schedule bounds the value from above; under gradient dominance the bound is explicit. If $f$ has $L$-Lipschitz gradient and satisfies the Polyak--{\L}ojasiewicz (gradient dominance) inequality $\tfrac12\|\nabla f(z)\|^2\ge\mu\,(f(z)-f^\star)$ for some $\mu>0$~\cite{Polyak1963,KarimiNutiniSchmidt2016}, as every $\mu$-strongly convex $f$ does, then the constant-step iteration $z_{k+1}=z_k-\tfrac1L\nabla f(z_k)$ reaches $\T_\varepsilon$ from observed gradients alone, with no inner minimization and no knowledge of $z^\star$, so
\begin{equation}
V(z_0)\le\Bigl\lceil\tfrac{L}{\mu}\log\bigl(2L\,(f(z_0)-f^\star)/\varepsilon^2\bigr)\Bigr\rceil.
\end{equation}
Reaching $\T_\varepsilon$ then places the iterate within $\varepsilon^2/(2\mu)$ of $f^\star$ and, for $\mu$-strongly convex $f$, within $\varepsilon/\mu$ of $z^\star$: on this class the intrinsic target certifies proximity to the minimizer, not only stationarity.
\end{example}

\section{The Optimality System of the Minimum-Time Problem}
\label{sec:solving}

We characterize the optimal solution by two routes. Dynamic programming yields the value function and a feedback law, a map assigning a coefficient choice to each state, over all initial conditions; the discrete Pontryagin minimum principle, applied at each fixed horizon, yields the optimal open-loop schedule from a given start. Off the quadratic the inner problem is nonconvex, so the conditions derived by either route are necessary only: they characterize an optimal schedule when one exists, and they do not by themselves decide whether any schedule reaches the target.

\subsection{Dynamic programming}
\label{subsec:dp}

Treating~\eqref{eq:value} as a minimum-time problem on the lifted state $(z,G)$, where $G=[\,g_0,\ \dots,\ g_{k-1}\,]$ collects the gradients observed so far, the principle of optimality gives the Bellman recursion
\begin{equation}
\label{eq:bellman}
V(z)=
\begin{cases}
0, & z\in\T_\varepsilon,\\[2pt]
1+\displaystyle\min_{c}\,V\!\big(z-c^\top G\big), & z\notin\T_\varepsilon,
\end{cases}
\end{equation}
where $c$ collects the corresponding coefficients; we abbreviate the value as $V(z)$, its dependence on the observed gradients suppressed to keep the notation light. The value function~\eqref{eq:bellman} is integer-valued and, by the principle of optimality, equals~\eqref{eq:value}: $V(z_0)$ is the minimum number of stages to reach $\T_\varepsilon$. The associated feedback law is
\begin{equation}
\label{eq:feedback}
\mu(z)\in\arg\min_{c}\,V\!\big(z-c^\top G\big).
\end{equation}
Dynamic programming returns a feedback law $\mu$ defined over all states, at the cost of solving~\eqref{eq:bellman} on the lifted state space, which is practical only in special low-dimensional cases; we use the dynamic-programming route conceptually, and the computations in this paper all run through the Pontryagin route.

\begin{remark}[The optimal feedback law]
\label{rem:feedback}
A coefficient $c$ is minimum-time optimal at $z$ exactly when it lands the successor $z-c^\top G$ in the set where the value equals $V(z)-1$, and more than one coefficient can do so; this is why~\eqref{eq:feedback} is set-valued. The number of steps from $z$ is therefore unique; more than one schedule can realize it. Selecting from~\eqref{eq:feedback} a coefficient of least magnitude gives a single-valued, piecewise smooth law. The selection jumps at states where two distinct landing points achieve the minimum in~\eqref{eq:bellman}, the selection switching from one landing point to the other. The law is defined at every state, whether or not an optimal trajectory passes there; at a state near the set where $V=\infty$ the target is barely reachable, and the selection can grow without bound, ever larger coefficients being needed to land in a set of finite value; any implementation of the feedback near that set inherits the blow-up as ill-conditioning.
\end{remark}

\subsection{The Pontryagin route}
\label{subsec:pmp}

Fix a horizon $K$ and define the inner problem
\begin{equation}
\label{eq:inner}
J^\star(K;z_0)=\inf_{\{c_{k,j}\}}\ f(z_K)\quad\text{s.t.}\quad\eqref{eq:fo-class},
\end{equation}
the best terminal value achievable in exactly $K$ steps, an infimum attained in every case computed in this paper: a fixed-horizon discrete optimal control problem, to which the discrete Pontryagin minimum principle applies~\cite{Naidu2003}, inside the outer minimum-time search over the horizon. The outer search presumes that some horizon reaches the target, the feasibility question posed in Sec.~\ref{sec:formulation}. Throughout this subsection $k$ indexes the stage, $j\le k$ the gradient $g_j=\nabla f(z_j)$ observed at stage $j$ and reused at stage $k$, and $p_{k+1}$ the costate adjoined to the update from $z_k$ to $z_{k+1}$. We adjoin the dynamics~\eqref{eq:fo-class} with these costates and form, at each stage, the Hamiltonian
\begin{equation}
\label{eq:ham}
H_k=\Big\langle p_{k+1},\ z_0-\sum_{j\le k}c_{k,j}\,\nabla f(z_j)\Big\rangle.
\end{equation}
Since the coefficient $c_{k,j}$ enters only the definition of $z_{k+1}$, with $\partial z_{k+1}/\partial c_{k,j}=-g_j$, the stationarity is
\begin{equation}
    \frac{\partial J}{\partial c_{k,j}}=-\inner{p_{k+1}}{g_j}=0,
\label{eq:stat}
\end{equation}
the costates are
\begin{equation}
    p_K=\nabla f(z_K),\quad
    p_m=\nabla^2 f(z_m)\!\!\sum_{r=m}^{K-1}\!(-c_{r,m})\,p_{r+1},
\label{eq:costate}
\end{equation}
where $p_m=\partial J/\partial z_m$ is the sensitivity of the terminal cost to the state, computed by~\eqref{eq:costate} as a backward sweep. The inner problem~\eqref{eq:inner} is a free terminal minimization, so the costate~\eqref{eq:costate} terminates at the cost gradient with no transversality multiplier; the outer horizon search carries the minimum-time structure, kept transverse by $\varepsilon>0$. Exact reaching, $\varepsilon=0$, forces $\nabla f(z_K)=0$ and collapses the terminal adjoint. Since the inner problem minimizes the function value~\eqref{eq:inner}, the exact-reaching value $V_0(z_0)$ is the first horizon at which the terminal gradient $\nabla f(z_K)$ can be driven to zero, and $V(z_0)$ its small-$\varepsilon$ counterpart; the function-value and gradient-norm stopping criteria differ for $\varepsilon>0$; they coincide as $\varepsilon\downarrow0$ when every reachable critical point is a global minimizer, as on the quadratic.

The recursion~\eqref{eq:costate} is non-local: $p_m$ depends on every later costate, not on $p_{m+1}$ alone, because the gradient $g_m$ observed at stage $m$ is reused at every later stage of~\eqref{eq:fo-class}. A state-space plant has a one-step adjoint recursion; here gradient reuse couples every stage, so the plant carries the whole gradient history rather than a current state. The terminal condition $p_K=\nabla f(z_K)$ with~\eqref{eq:stat} states that the final gradient is orthogonal to every previously observed gradient, the Galerkin condition of conjugate-direction methods and, equivalently, the stationarity certificate for~\eqref{eq:target}: it vanishes once the explored subspace meets the critical set.

The Hessian in~\eqref{eq:costate} does not make the benchmark second-order. The residuals the benchmark solves are the terminal inner products $\inner{g_K}{g_j}$ of~\eqref{eq:stat} with $p_K=g_K$, evaluated from oracle outputs alone; no Hessian is formed, its dependence absorbed into differences of observed gradients, exactly as conjugate gradient realizes conjugacy through exact line search.

To verify~\eqref{eq:costate}, assume $f\in C^2$ and induct downward on $m$; the recursion then holds at every horizon. For $m=K$, $J=f(z_K)$ gives $p_K=\nabla f(z_K)$. The node $z_m$ enters the dynamics~\eqref{eq:fo-class} only through its gradient $g_m=\nabla f(z_m)$, which appears in each later iterate $z_{r+1}$, $r=m,\dots,K-1$, with coefficient $-c_{r,m}$ and nowhere else; a perturbation $\delta z_m$ changes $g_m$ by $\nabla^2 f(z_m)\,\delta z_m$ and propagates to each later node as $\delta z_{r+1}=-c_{r,m}\nabla^2 f(z_m)\,\delta z_m$. Summing the downstream sensitivities $p_{r+1}=\partial J/\partial z_{r+1}$ by the chain rule, and using symmetry of the Hessian,
\begin{equation}
p_m=\sum_{r=m}^{K-1}\Bigl(\frac{\partial z_{r+1}}{\partial z_m}\Bigr)^{\!\top}\!p_{r+1}
=\nabla^2 f(z_m)\sum_{r=m}^{K-1}(-c_{r,m})\,p_{r+1},
\label{eq:costateproof}
\end{equation}
a perturbation $\delta z_m$ reaches $z_{r+1}$ only through the single direct term $-c_{r,m}g_m$, so the chain rule counts each path exactly once; this is~\eqref{eq:costate}.

\subsection{The value as a computable benchmark}
\label{subsec:benchmark}
Every benchmark value reported in this paper is computed by Algorithm~\ref{alg:benchmark}: an ascending sweep over the horizon, with a nonlinear least-squares solve of the optimality system at each fixed horizon, so that the first certified horizon is the minimum, the value itself. The sweep may start at any $K_{\min}$ known to satisfy $K_{\min}\le V(z_0)$ without affecting this reading, and from a larger $K_{\min}$ the returned horizon is still a certified upper bound; $K_{\min}=1$ is always safe and is used throughout this paper, the small-horizon solves being cheap, the coefficient array at horizon $K$ having $K(K+1)/2$ entries. The residuals are the terminal stationarity inner products $\langle g_K,g_j\rangle$, evaluated from observed gradients alone (Sec.~\ref{subsec:pmp}).

\begin{algorithm}[t]
\caption{Benchmark computation of the value}
\label{alg:benchmark}
\begin{algorithmic}[1]
\REQUIRE oracle access to $f$ and $\nabla f$; start $z_0$; tolerance $\varepsilon\ge0$; horizons $1\le K_{\min}\le K_{\max}$
\IF{$\|\nabla f(z_0)\|\le\varepsilon$}
\RETURN the horizon $0$ and the empty schedule
\ENDIF
\FOR{$K=K_{\min},K_{\min}+1,\dots,K_{\max}$}
\STATE solve~\eqref{eq:stat} and~\eqref{eq:costate}, under the dynamics~\eqref{eq:fo-class}, as a nonlinear least-squares problem in the coefficient array $\{c_{k,j}\}_{0\le j\le k<K}$, with residuals $\inner{g_K}{g_j}$, $j=0,\dots,K-1$ (observed gradients only; no Hessian is formed)
\IF{$\|\nabla f(z_K)\|\le\varepsilon$ at the solution}
\RETURN the horizon $K$ and the realizing schedule $\{c_{k,j}\}$
\ENDIF
\ENDFOR
\RETURN no certificate up to $K_{\max}$
\end{algorithmic}
\end{algorithm}

The horizon returned is the benchmark and the schedule its realization: on the strongly convex quadratic the inner problem is convex, the solve reproduces the conjugate gradient iterates, and, as the next section shows, the value is exact; off the quadratic the solve returns an extremal and a certified upper bound on $V_0(z_0)$. Every runnable first-order method, whether its coefficients are fixed in advance or chosen from observed data, executes one coefficient schedule of~\eqref{eq:fo-class}, so its oracle count on the instance is bounded below by $V(z_0)$: the value is the best per-instance performance achievable in the class, and the gap a given method leaves to it is measured on the same instance. The benchmark is clairvoyant rather than runnable, its coefficients chosen with knowledge of the instance; on the quadratic it is nonetheless attained causally by conjugate gradient, one evaluation per step. Off the quadratic the optimal coefficients are not determined by the observed gradients alone: realizing the count would mean minimizing $f$ over the affine slice~\eqref{eq:reachspan}, a low-dimensional but generally nonconvex problem that itself consumes oracle queries; a causal method instead settles for a query count set by the conditioning, as in the gradient-descent ceiling of Example~\ref{ex:pl}. We therefore claim no transparent realization of $V$ beyond the convex regime, and take as the fair runnable comparison a fixed-coefficient method such as Nesterov's or conjugate gradient.

\section{The Quadratic Case: Conjugate Gradient Emerges from Optimality}
\label{sec:necessary}
\label{subsec:convex}

This section shows that, on a strongly convex quadratic, the conjugate gradient iterates emerge from the optimality system of Sec.~\ref{sec:solving} as the minimum-time trajectory, and that the exact-reaching value equals the controllability index of a single-input pair (Thm.~\ref{thm:selection}). For quadratics, the identification of the value with a controllability index is classical, the Krylov space being the controllable subspace~\cite{Kalman1963,Sontag1998}; what the optimality system adds is that the necessary conditions hold along the conjugate gradient iterates alone, where every span method is merely admissible. In this case the Pontryagin construction is globally optimal and the value is exact; the general case is addressed in the sections that follow.

Let
\begin{equation}
\label{eq:quadratic}
f(z)=\tfrac12(z-z^\star)^\top Q(z-z^\star),\qquad Q\succ0,
\end{equation}
and write $w_k=z_k-z^\star$, so that $\nabla f(z_k)=Qw_k$. We build up the picture by dimension.

\begin{definition}[Reachable subspace and grade]
\label{def:reach}
For $Q\succ0$ and $w_0\in\R^n$, the order-$K$ Krylov subspace generated by $Q$ and $Qw_0$ is
\begin{equation}
\mathcal{K}_K(Q,Qw_0)=\operatorname{span}\{Qw_0,\,Q^2w_0,\,\dots,\,Q^K w_0\},
\label{eq:reachsub}
\end{equation}
and the reachable subspace of the single-input pair $(Q,Qw_0)$ is its limit $\mathcal{K}_n(Q,Qw_0)$. The grade of $w_0$ with respect to $Q$ is the dimension of $\mathcal{K}_n(Q,Qw_0)$~\cite{Saad2003}, equivalently the degree of the minimal polynomial of $Q$ acting on $w_0$.
\end{definition}

The same subspace arises twice. A Krylov method searches $\mathcal{K}_K(Q,Qw_0)$, the space spanned by repeated applications of $Q$ to the first gradient. A single-input linear system with system matrix $Q$ and input direction $Qw_0$ reaches, in $K$ steps, exactly $\mathcal{K}_K(Q,Qw_0)$. On the quadratic, the first-order iteration~\eqref{eq:fo-class} is both constructions at once, and this identification carries the controllability toolbox to the optimization problem. The pair is controllable when the grade is $n$, and not controllable when $w_0$ lies in a proper $Q$-invariant subspace, where the grade falls below $n$. We call the grade the controllability index of the pair and use it throughout, defined whether or not the pair is controllable.

In the scalar case $n=1$, the single gradient $g_0=Qw_0$ spans the line through $z_0$ in the direction of $z^\star$, so one coefficient reaches $z^\star$ exactly: the exact-reaching value is $V_0(z_0)=1$ for every $z_0\ne z^\star$.
In the planar case $n=2$, if $w_0$ is an eigenvector of $Q$, then $g_0\parallel w_0$ and one step suffices, $V_0(z_0)=1$. Otherwise $g_0$ and $g_1$ are linearly independent, span $\R^2$, and two coefficients place $z_2=z^\star$, so $V_0(z_0)=2$.

In the general case, from~\eqref{eq:fo-class},
\begin{equation}
\label{eq:krylov}
w_K\in w_0+\mathcal{K}_K(Q,Qw_0),
\end{equation}
the shifted Krylov subspace, which contains $w_0$, whereas the Krylov subspace~\eqref{eq:reachsub} is spanned by the images $Q^jw_0$ and need not contain $w_0$ itself. Reaching $z^\star$ (i.e. $w_K=0$) is possible iff $w_0\in\mathcal{K}_K(Q,Qw_0)$, so $V_0(z_0)$ is
\begin{equation}
\label{eq:grade}
V_0(z_0)=\min\{K:\ w_0\in\mathcal{K}_K(Q,Qw_0)\},
\end{equation}
the grade of $w_0$ with respect to $Q$, equivalently the degree of the minimal polynomial of $Q$ acting on $w_0$. Since the only critical point is $z^\star$, here $\Crit(f)=\{z^\star\}$; as $z^\star\in\T_\varepsilon$ we have $V(z_0)\le V_0(z_0)$, with equality for small $\varepsilon$ (Sec.~\ref{sec:formulation}), and throughout the quadratic case we report this exact-reaching value. By the Cayley--Hamilton theorem this is at most $n$, and the conjugate gradient method~\cite{HestenesStiefel1952} attains it: its terminal Galerkin condition is exactly~\eqref{eq:stat} with $p_K=g_K$, and its conjugate directions realize the span~\eqref{eq:krylov} from observed gradients alone.

Equation~\eqref{eq:grade} is the controllability index of the single-input linear system with system matrix $Q$ and input direction $Qw_0$~\cite{Kalman1963}. For a single-input pair $(A,b)$ the controllability matrix, and its instance here, are
\begin{subequations}
\label{eq:ctrbmat}
\begin{align}
\mathcal{C}(A,b)&=[\,b,\ Ab,\ \dots,\ A^{n-1}b\,],\\
\mathcal{C}(Q,Qw_0)&=[\,Qw_0,\ Q^2w_0,\ \dots,\ Q^{n}w_0\,],
\end{align}
\end{subequations}
whose columns span the reachable subspace $\mathcal{K}_n(Q,Qw_0)$; $V_0(z_0)$ is the smallest $k$ at which the leading $k$ columns attain that dimension, equivalently the grade $d_0$. The reachable subspace is thus the controllable subspace, and $V_0(z_0)$ is this controllability index. An eigenvector start gives grade one, a start in a $d$-dimensional invariant subspace gives grade $d$, and a generic start gives grade $n$. The value is basis-independent, hence unchanged by rotating $Q$ to a non-diagonal, i.e. non-separable, form. The first-order span, the Krylov subspace, and the controllable subspace thus coincide, and the minimum-time value is their common dimension count.

On the quadratic $g_j=Qw_j$, so the reachable span $S_K(z_0)$ of~\eqref{eq:reachspan} is $\operatorname{span}\{Qw_0,\dots,Qw_{K-1}\}$.
The optimality conditions ask for a terminal gradient orthogonal to every earlier one. Conjugate gradient builds gradients with exactly this property: each step performs an exact line search along a direction kept $Q$-conjugate to its predecessors, which makes the residuals, the gradients up to sign, mutually orthogonal. Writing $r_k=-g_k=-\nabla f(z_k)$ for the residual and $u_k$ for the directions, the recursion is
\begin{subequations}
\label{eq:cg}
\begin{align}
u_0&=r_0=-g_0,\label{eq:cg-init}\\
\alpha_k&=\inner{r_k}{r_k}/\inner{u_k}{Qu_k},\label{eq:cg-alpha}\\
z_{k+1}&=z_k+\alpha_k u_k,\label{eq:cg-z}\\
r_{k+1}&=r_k-\alpha_k Qu_k,\label{eq:cg-r}\\
\beta_k&=\inner{r_{k+1}}{r_{k+1}}/\inner{r_k}{r_k},\label{eq:cg-beta}\\
u_{k+1}&=r_{k+1}+\beta_k u_k.\label{eq:cg-dir}
\end{align}
\end{subequations}
In the recursion, \eqref{eq:cg-alpha} is the exact line-search coefficient, \eqref{eq:cg-z} and \eqref{eq:cg-r} advance the iterate and the residual, and \eqref{eq:cg-beta} and \eqref{eq:cg-dir} generate the next conjugate direction from the new residual. The mutual orthogonality $\inner{r_i}{r_j}=0$ for $i\neq j$ is the residual-orthogonality (Galerkin) property of Hestenes and Stiefel~\cite{HestenesStiefel1952}; equivalently $z_k$ minimizes $f$ over the affine Krylov space $z_0+\mathcal{K}_k(Q,Qw_0)$. The next two lemmas, and the theorem they yield, show that the minimum-time optimality system of Sec.~\ref{subsec:pmp} is satisfied by exactly these iterates, rather than by gradient descent or any other span method.

\begin{lemma}[Galerkin characterization]
\label{lem:galerkin}
For the strongly convex quadratic~\eqref{eq:quadratic}, with $g_k=Qw_k$ and $d_0$ the grade of $w_0$ in~\eqref{eq:grade}, fix a horizon $K\le d_0$ and a trajectory of~\eqref{eq:fo-class} with $\dim S_K(z_0)=K$. The terminal stationarity in~\eqref{eq:stat}, with the terminal costate $p_K=g_K$ of~\eqref{eq:costate}, reads
\begin{equation}
\inner{g_K}{g_j}=0,\qquad j=0,\dots,K-1,
\label{eq:galerkin}
\end{equation}
the orthogonality of the terminal residual to the reachable span $S_K(z_0)$, and it is necessary and sufficient for $z_K$ to minimize the inner objective~\eqref{eq:inner} over the slice $z_0-S_K(z_0)$.
\end{lemma}
\begin{proof}
On the quadratic the inner objective~\eqref{eq:inner} at horizon $K$ is $\tfrac12 w_K^\top Q w_K$, with
\begin{equation}
\label{eq:wK}
w_K=w_0-\sum_{j<K}c_{K-1,j}\,g_j
\end{equation}
ranging over $w_0+S_K(z_0)$. Collecting the spanning gradients as $G_K=[\,g_0,\dots,g_{K-1}\,]$, the objective is $\tfrac12(w_0-G_Kc)^\top Q(w_0-G_Kc)$ in the terminal coefficients $c=(c_{K-1,j})_j$, with Hessian $G_K^\top Q G_K$. Since $Q\succ0$ and $\dim S_K(z_0)=K$ makes $G_K$ of full column rank, the objective is strictly convex, so its unique minimizer is the stationary point $G_K^\top Q(w_0-G_Kc)=0$, that is $G_K^\top g_K=0$, which is~\eqref{eq:galerkin}.
\end{proof}

\begin{lemma}[Residual orthogonality]
\label{lem:cggalerkin}
For the strongly convex quadratic~\eqref{eq:quadratic}, with $d_0$ the grade of $w_0$ in~\eqref{eq:grade}, the conjugate gradient residuals are mutually orthogonal, $\inner{g_i}{g_j}=0$ for $i\neq j$, so~\eqref{eq:galerkin} holds at every horizon along a single trajectory: one recursion is inner-optimal for all $K\le d_0$ and reaches $z^\star$ at $K=d_0$.
\end{lemma}
\begin{proof}
The mutual orthogonality of the residuals is the Hestenes--Stiefel property recalled at~\eqref{eq:cg}~\cite{HestenesStiefel1952}, and it is~\eqref{eq:galerkin} at every horizon simultaneously; with Lemma~\ref{lem:galerkin}, the conjugate gradient iterates are inner-optimal for every $K\le d_0$. At $K=d_0$ the constant term of the minimal polynomial of $Q$ acting on $w_0$ is nonzero because $Q\succ0$, so $w_0\in\mathcal{K}_{d_0}(Q,Qw_0)$ by~\eqref{eq:grade}: the projection is exact, $w_{d_0}=0$ and $g_{d_0}=0$.
\end{proof}

\begin{theorem}[Conjugate gradient emerges from optimality]
\label{thm:selection}
For the strongly convex quadratic~\eqref{eq:quadratic}, with $d_0$ the grade of $w_0$ in~\eqref{eq:grade}, the discrete Pontryagin conditions of Sec.~\ref{subsec:pmp} yield the conjugate gradient iterates at each horizon $K\le d_0$, with value $V_0(z_0)=d_0$ of~\eqref{eq:grade}.
\end{theorem}
\begin{proof}
By Lemma~\ref{lem:galerkin} the optimality conditions at horizon $K$ are the Galerkin condition~\eqref{eq:galerkin}, and by Lemma~\ref{lem:cggalerkin} the conjugate gradient iterates satisfy it at every horizon $K\le d_0$, reaching $z^\star$ at $K=d_0$; that no schedule reaches earlier, and hence $V_0(z_0)=d_0$, is~\eqref{eq:grade}.
\end{proof}

The costate recursion~\eqref{eq:costate} thus does real work on the
quadratic: its terminal condition $p_K=g_K$ turns the stationarity
\eqref{eq:stat} into the Galerkin condition~\eqref{eq:galerkin}, and the necessary conditions hold at every horizon along exactly one iterate sequence, the conjugate gradient iterates. Condition~\eqref{eq:galerkin} identifies $z_K$ as the energy-norm projection of $z^\star$ onto the affine Krylov space $z_0+\mathcal{K}_K(Q,Qw_0)$, the $K$th conjugate gradient iterate, and the value $V_0(z_0)=d_0$ follows from the optimality system rather than by separate construction. 
The Pontryagin and dynamic-programming routes therefore agree on the quadratic, the former producing the optimal open-loop schedule that the latter attains pointwise. 
The schedule need not be unique: different trajectories with $\dim S_K(z_0)=K$ generate the same Krylov space and the same inner optimum, so the optimality system pins down the iterates rather than a unique coefficient array.

\begin{remark}[Count versus rate]
\label{rem:minimax}
Fix a class $\mathcal F$ of objectives and a start $z_0$, write $V(f,z_0)$ for the value of Definition~\ref{def:value} with the objective explicit, and let $W_{\mathcal F}(\varepsilon)$ be the worst-case oracle complexity over $\mathcal F$, to tolerance $\varepsilon$, of a method optimal in the worst case, the quantity performance estimation computes~\cite{DroriTeboulle2014,TaylorHendrickxGlineur2017,KimFessler2016}. Such a method reaches $\T_\varepsilon$ on every $f\in\mathcal F$ within $W_{\mathcal F}(\varepsilon)$ queries, so
\begin{equation}
\label{eq:minimax}
\sup_{f\in\mathcal F} V(f,z_0)\ \le\ W_{\mathcal F}(\varepsilon).
\end{equation}
The inequality is elementary; on quadratics its looseness is quantitative. The exact-reaching value~\eqref{eq:grade} is the grade, $d_0\le\nu\le n$ with $\nu$ the number of distinct eigenvalues of $Q$, independent of the condition number $\kappa$; the query count of a method governed by a fixed contraction factor tuned to $[\lambda_{\min},\lambda_{\max}]$, such as gradient descent, the heavy-ball method, or a worst-case rate-optimal accelerated method, is $O(\sqrt{\kappa}\log(1/\varepsilon))$ by the standard Chebyshev argument~\cite{GolubVanLoan2013,Saad2003}, tight in the worst case over quadratics of condition number $\kappa$~\cite{NemirovskiYudin1983,Nesterov2018}. With $\nu$ fixed and $\kappa\to\infty$ the two diverge, and the left side of~\eqref{eq:minimax} over the clustered class is of order $\nu$ while the right is of order $\sqrt{\kappa}\log(1/\varepsilon)$. On a $\nu=2$ instance with $\kappa=100$ the value is $2$, while the triple-momentum method~\cite{VanScoyFreemanLynch2018} and gradient descent take about $88$ and several hundred queries, the former the rate-bound count $\lceil\log(1/\varepsilon)/(2\log(1/\rho))\rceil$ at $\rho=1-1/\sqrt{\kappa}$; since the value depends only on the number of distinct eigenvalues and the rate only on $\kappa$, the count is shared by the $2\times2$ conditioned quadratic of~\cite{GramlichEbenbauerScherer2022} and the $n=50$, $\nu=2$ family of Fig.~\ref{fig:gap}. The figure's stopping rule is on the function gap, which agrees with the gradient-norm target as $\varepsilon\downarrow0$, both targets reaching the exact value $d_0$. The gap is a property of the spectrum, not the dimension: on the Nemirovski--Yudin tridiagonal instance~\cite{NemirovskiYudin1983}, the canonical case on which the accelerated bound is tight, the condition number is tied to the dimension, $\kappa\approx4n^2/\pi^2$, and the value $n$ sits only a fixed $\log(1/\varepsilon)$ factor below the bound. This does not contradict the worst-case bound, which governs the worst instance of the class and is saturated by quadratics; it quantifies how loose the bound is on a fixed instance. The rate and the count read the same first-order span~\eqref{eq:fo-class}: the rate is its yield over the hardest instance of a condition-number class, the value its yield on the instance at hand.
\end{remark}

\begin{figure}[t]
\centering
\includegraphics[width=0.99\columnwidth]{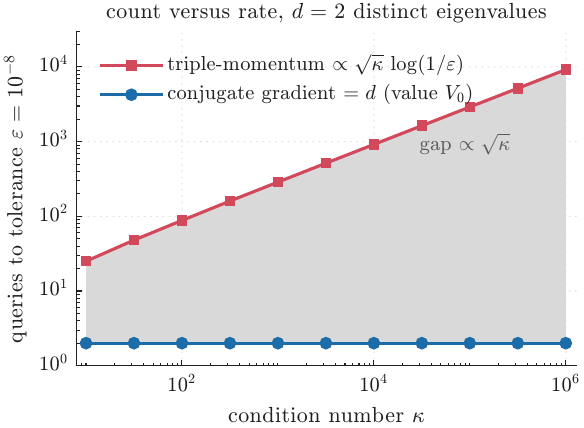}
\caption{Count versus rate on a strongly convex quadratic with $\nu=2$ distinct eigenvalues ($1$ and $\kappa$, $n=50$, generic start, $\varepsilon=10^{-8}$ on the relative function gap): conjugate gradient terminates at the grade $d_0=2$, independent of $\kappa$, while the triple-momentum method~\cite{VanScoyFreemanLynch2018} takes $O(\sqrt{\kappa}\log(1/\varepsilon))$ queries. With $\nu$ fixed the gap diverges as $\sqrt{\kappa}$.}
\label{fig:gap}
\end{figure}

\section{Reachability and Controllability Beyond the Quadratic}
\label{sec:reach}
\label{subsec:controllability}

The quadratic case organized the value function through a fixed linear system, for which the reachable span is a Krylov subspace and the value is a controllability index. We now isolate the structure that survives when $f$ is not quadratic, where no fixed system matrix and no conjugate gradient guarantee are available, in two steps: a reachability test for feasibility and a curvature shortcut that can lower the value.

\subsection{The reachability test}
\label{subsec:reachtest}

This subsection certifies which instances admit a finite reaching schedule, and in how few steps; that certificate is what the construction of Sec.~\ref{sec:solving} presupposes. Reachability is the feasibility condition of the reaching problem, the analogue for this finite-horizon problem of controllability for the linear-quadratic regulator; unlike controllability it is target-specific, so it can hold for one target while failing for another. Beyond the quadratic the inner minimization is generally nonconvex, so the Pontryagin stationarity is necessary but not sufficient and need not single out the inner optimum; a nonlinear least-squares solve of the optimality system then returns an extremal and a certified upper bound on the value, not a guaranteed optimum. We therefore characterize reachability directly, by whether the reachable affine slice meets the critical set, a condition that is exact and requires no inner minimization. Every iterate of~\eqref{eq:fo-class} lies in $z_0-S_K(z_0)$, the affine slice of the reachable span~\eqref{eq:reachspan}, so reaching a critical point is a question of whether the slice meets the critical set, and reaching $z^\star$ of whether $z^\star-z_0$ lies in the span.

Define the reachable gradient span from $z_0$,
\begin{equation}
\label{eq:span}
\Grad(z_0)=\operatorname{span}\big\{\nabla f(z):\ z\ \text{reachable from}\ z_0\ \text{under}\ \eqref{eq:fo-class}\big\}.
\end{equation}
Every iterate lies in $z_0-\Grad(z_0)$, and the dimension of the $K$-step reachable set is at most $K$.

\begin{definition}[Stationarity-reachable instance]
\label{def:reachable}
The instance $(f,z_0)$ is stationarity-reachable if $\nabla f(z_K)=0$ for some finite horizon $K$ and some schedule of~\eqref{eq:fo-class}.
\end{definition}

Written out, the exact-reaching value is
\begin{equation}
V_0(z_0)=\min\bigl\{K:\ \exists\,\{c_{k,j}\}\ \text{with}\ \nabla f(z_K)=0\bigr\},
\label{eq:V0def}
\end{equation}
finite exactly on the stationarity-reachable instances, and $V(z_0)\le V_0(z_0)$ since every critical point lies in $\T_\varepsilon$. The engine of the test is a growth dichotomy for the reachable span.

\begin{lemma}[Span growth]
\label{lem:spangrowth}
Let $f\in C^2$, let $A_K=z_0-S_K(z_0)$, and let $P_K$ be the orthogonal projection onto $S_K(z_0)^\perp$. A continuation with $\dim S_{K+1}(z_0)=\dim S_K(z_0)+1$ exists unless
\begin{equation}
P_K\,\nabla^2 f(z)\big|_{S_K(z_0)}=0\quad\text{for all }z\in A_K,
\label{eq:invariance}
\end{equation}
that is, unless $S_K(z_0)$ is $\nabla^2 f$-invariant along $A_K$; and in the invariant case every continuation keeps $z_k\in A_K$ and $g_k\in S_K(z_0)$, so the span never grows again.
\end{lemma}
\begin{proof}
At stage $K$ the terminal coefficients of~\eqref{eq:fo-class} are free, so the next query point $z_K$ may be placed at any $z\in A_K$; the span grows iff $g_K=\nabla f(z)\notin S_K(z_0)$ for some such $z$, that is, iff $P_K\nabla f\not\equiv0$ on $A_K$. We claim $P_K\nabla f\equiv0$ on $A_K$ if and only if~\eqref{eq:invariance} holds. For $z\in A_K$ and $v\in S_K(z_0)$, the line $z+tv$ stays in $A_K$ and
\begin{equation}
\tfrac{d}{dt}\,P_K\nabla f(z+tv)\big|_{t=0}=P_K\,\nabla^2 f(z)\,v.
\label{eq:dirderiv}
\end{equation}
If $P_K\nabla f\equiv0$ on $A_K$, all such derivatives vanish, which is~\eqref{eq:invariance}. Conversely, if~\eqref{eq:invariance} holds, then by~\eqref{eq:dirderiv} the map $P_K\nabla f$ is constant along every segment of the affine slice $A_K$, hence constant on $A_K$, and its value at $z_0\in A_K$ is $P_Kg_0=0$ since $g_0\in S_K(z_0)$; so $P_K\nabla f\equiv0$. In the invariant case every point of $A_K$ has gradient in $S_K(z_0)$, so every continuation keeps $z_k\in A_K$ and $g_k\in S_K(z_0)$, and the span never grows again.
\end{proof}

\begin{theorem}[Reachability test]
\label{thm:controllability}
Let $f\in C^2$ with global minimizer $z^\star$, and let $z_0\notin\T_\varepsilon$ with $\nabla f(z_0)\neq0$.
\begin{enumerate}
\item[(a)] Along any trajectory of~\eqref{eq:fo-class}, a critical point is achievable by a choice of the terminal coefficients if and only if the affine slice $z_0-S_K(z_0)$ meets the critical set, so
\begin{equation}
\begin{aligned}
V_0(z_0)=\min\bigl\{K:\ &(z_0-S_K(z_0))\cap\Crit(f)\neq\emptyset\\
&\text{for some trajectory}\bigr\},
\end{aligned}
\label{eq:Vspan}
\end{equation}
with $V_0(z_0)=\infty$ if no such $K$ exists.
\item[(b)] If no slice reachable from $z_0$ is $\nabla^2 f$-invariant while disjoint from the critical set, then $V_0(z_0)\le n$. Conversely, $V_0(z_0)=\infty$ if and only if every trajectory from $z_0$ freezes at an invariant slice that contains no critical point.
\end{enumerate}
\end{theorem}
\begin{proof}
(a) Fix a trajectory $z_0,\dots,z_{K-1}$ with observed gradients $g_0,\dots,g_{K-1}$ spanning $S_K(z_0)$. The terminal iterate is 
\begin{equation}
    z_K=z_0-\sum_{j<K}c_{K-1,j}\,g_j
\end{equation}
with the coefficients free in $\R^K$, so the set of achievable $z_K$ is exactly the affine slice $A_K=z_0-S_K(z_0)$, and a critical point is achievable if and only if some $z_c\in\Crit(f)$ lies in $A_K$, that is $z_c-z_0\in S_K(z_0)$. Taking the minimum over $K$ and over trajectories gives~\eqref{eq:Vspan}: the right side is achievable by the trajectory realizing it, and any schedule landing on a critical point $z_c$ exhibits 
\begin{equation}
    z_c-z_0=\sum_j c_{K-1,j}g_j\in S_K(z_0)
\end{equation} 
along its own trajectory.

(b) For the first claim, grow the span greedily: starting from $S_1=\operatorname{span}\{g_0\}$, as long as the current slice is not invariant, Lemma~\ref{lem:spangrowth} supplies a continuation raising the dimension by one. The process halts either when the span reaches $\R^n$, after at most $n$ steps, in which case the slice is all of $\R^n$ and contains every critical point, so part~(a) places $z_K$ on one with $K\le n$; or when it freezes at an invariant slice, which by hypothesis meets the critical set, and part~(a) places the iterate on a critical point within it, again with $K\le\dim S_K\le n$. For the converse, if some trajectory's slice meets the critical set then $V_0<\infty$ by part~(a); if every trajectory freezes at an invariant slice containing no critical point, then by the last clause of Lemma~\ref{lem:spangrowth} no continuation of any trajectory ever meets the critical set, and $V_0=\infty$.
\end{proof}

For real-analytic objectives the reachable dimension is generic. Proposition~\ref{prop:generic} makes this precise; its proof is deferred to Appendix~\ref{app:thm1}.

\begin{proposition}[Generic reachability]
\label{prop:generic}
If no proper nonzero subspace is $\nabla^2 f$-invariant along a reachable slice, then the span never freezes and $V_0(z_0)\le n$ for every $z_0$ with $\nabla f(z_0)\neq0$. For real-analytic $f$, $\dim\Grad(z_0)$ equals a fixed value $n_\star$ for every $z_0$ outside a set of Lebesgue measure zero, and $n_\star=n$ unless some proper nonzero subspace is $\nabla^2 f$-invariant along a reachable slice.
\end{proposition}

The distinction between $V$ and $V_0$ is the tolerance: $V$ reaches the gradient-tolerance set $\T_\varepsilon$ and $V_0$ a critical point exactly, so $V\le V_0$ always, with the gap closing as $\varepsilon\downarrow0$. The characterization above is exact for $V_0$ and an upper bound for $V$. The quadratic statements of Sec.~\ref{subsec:convex} are exact-termination statements, stated for $V_0$, where the critical set is the single point $z^\star$, and they coincide with $V$ for every $\varepsilon$ below the smallest nonzero gradient gap along the trajectory. The tolerance $\varepsilon>0$ is needed only to keep the target transverse in the optimal-control formulation; the characterizations here are algebraic and apply to exact reaching.

Theorem~\ref{thm:controllability} is the nonlinear counterpart of the rank test: the constant controllability matrix~\eqref{eq:ctrbmat} of the quadratic case is replaced by the Hessian-generated span along the reachable slice, and the invariance condition~\eqref{eq:invariance} plays the role of the accessibility rank condition for the lifted state $(z_k,S_k)$ formed by the iterate and the span of observed gradients, the first-order analogue of the bracket-generating condition of nonlinear control~\cite{JakubczykSontag1990,AlbertiniSontag1993}. The analogy is structural rather than literal: those results characterize accessibility of an invertible state-space system by a Lie-algebra rank condition, whereas~\eqref{eq:invariance} is a single Hessian-image condition on a plant with memory. Its two forms mirror the two classical ones: the span-growth form, Thm.~\ref{thm:controllability} built on Lemma~\ref{lem:spangrowth}, is the rank condition that the reachable span capture $z^\star-z_0$, and the invariance form~\eqref{eq:invariance} is the invariant-subspace condition that no proper $\nabla^2 f$-invariant slice exclude it, the discrete analogue of the Popov--Belevitch--Hautus characterization~\cite{Sontag1998}; on the quadratic~\eqref{eq:quadratic}, where the invariance condition~\eqref{eq:invariance} is $Q$-invariance, both reduce to controllability of $(Q,Qw_0)$. Unlike the rank condition, however, the test is not a single finite rank evaluated in advance: off the quadratic no Cayley--Hamilton closure caps the span, so the condition is a property of the Hessian field $\{\nabla^2 f(z)\}$ along the slice. For real-analytic $f$ the obstruction either holds identically, giving a genuine invariant subspace and a frozen span, or holds only on a measure-zero set, off which the span reaches $\R^n$ within $n$ steps as in Proposition~\ref{prop:generic}. The clean sufficient condition for uniform reachability is that the Hessian field admit no proper common invariant subspace, the nonlinear analogue of a controllable pair.

A complementary sufficient condition makes invariant slices harmless rather than absent.

\begin{lemma}[Frozen slices and coercivity]
\label{lem:freeze}
Let $f\in C^2$. If a reachable slice $A_K=z_0-S_K(z_0)$ is $\nabla^2 f$-invariant in the sense of~\eqref{eq:invariance} and $f$ attains its infimum over $A_K$, then $A_K$ contains a critical point of $f$. In particular, if $\nabla f(z_0)\neq0$ and every slice reachable from $z_0$ lies in a subspace $\mathcal{W}$ on which $f$ is coercive (radially unbounded), then no reachable slice is $\nabla^2 f$-invariant while disjoint from the critical set, and $V_0(z_0)\le\dim\mathcal{W}$.
\end{lemma}
\begin{proof}
The invariance~\eqref{eq:invariance} gives $\nabla f(z)\in S_K(z_0)$ for every $z\in A_K$, as in the proof of Lemma~\ref{lem:spangrowth}. If $z_m$ attains the infimum of $f$ over the affine set $A_K$, whose direction space is $S_K(z_0)$, then $\nabla f(z_m)\perp S_K(z_0)$; hence $\nabla f(z_m)=0$ and $z_m\in\Crit(f)\cap A_K$. For the second claim, every reachable slice is a closed affine subset of $\mathcal{W}$, so its infimum is attained by coercivity and every invariant slice meets the critical set. The greedy growth in the proof of Thm.~\ref{thm:controllability}(b), confined to $\mathcal{W}$, halts within $\dim\mathcal{W}$ steps at an invariant slice: a halting slice exists because the full slice $z_0-\mathcal{W}$ is itself invariant, the slice continued through any of its points remaining in $\mathcal{W}$. Theorem~\ref{thm:controllability}(a) then places the iterate on a critical point within the halting slice.
\end{proof}

\begin{corollary}[Generalized linear models]
\label{cor:glm}
Let $f(z)=\sum_{i=1}^{m}\ell_i(a_i^\top z)$ with each $\ell_i\in C^2$ and data matrix $A=[a_1,\dots,a_m]^\top$, and suppose $f$ is coercive on the row space $\mathcal{R}=\operatorname{row}(A)$. For any start $z_0\in\mathcal{R}$, for instance $z_0=0$,
\begin{equation}
V(z_0)\le V_0(z_0)\le\operatorname{rank}(A),
\label{eq:glm}
\end{equation}
independent of the ambient dimension $n$: the per-instance count is set by the rank of the data, not by $n$.
\end{corollary}
\begin{proof}
Every gradient $\nabla f(z)=\sum_i \ell_i'(a_i^\top z)\,a_i$ lies in $\mathcal{R}$, and $\nabla^2 f(z)=A^\top\operatorname{diag}\bigl(\ell_i''(a_i^\top z)\bigr)A$ maps $\R^n$ into $\mathcal{R}$, so every slice reachable from $z_0\in\mathcal{R}$ lies in $\mathcal{R}$ and $\mathcal{R}$ is $\nabla^2 f$-invariant. Since $f$ is coercive on $\mathcal{R}$, Lemma~\ref{lem:freeze} gives $V_0(z_0)\le\dim\mathcal{R}=\operatorname{rank}(A)$, and $V\le V_0$ since a critical point lies in $\T_\varepsilon$. As $f$ is constant on the fibres $z+\ker A$, its critical set is $z^\star+\ker A$, an affine subspace meeting $\mathcal{R}$ only at $z^\star$; when $f|_{\mathcal{R}}$ is strictly convex, $z^\star$ is the unique minimizer on $\mathcal{R}$ and the minimum-norm minimizer of $f$.
\end{proof}

\begin{example}[Rank-two logistic regression]
\label{ex:logistic}
Take the logistic loss 
\begin{equation}
    f(z)=\sum_{i=1}^m\log(1+e^{-y_i a_i^\top z})
    \label{eq:log_loss}
\end{equation}
on a rank-two design $A\in\R^{m\times n}$, $m=300$, $n=200$, with orthonormal row space and in-row-space condition number $\sqrt{213}\approx14.6$. The labels are $y_i=\operatorname{sign}(\rho_i-\sigma(s_i^\top u_{\mathrm{true}}))$ with $\rho_i$ uniform on $[0,1]$, $s_i\in\R^2$ the coordinates of $a_i$ in an orthonormal basis of $\operatorname{row}(A)$, and $u_{\mathrm{true}}=(0.9,0)$, non-separable; the start is $z_0=0\in\mathcal{R}$, $\mathcal{R}=\operatorname{row}(A)$. Restricted to $\mathcal{R}$, $f$ is strongly convex with certified restricted condition number $\kappa\approx213$, the square of the in-row-space condition number because the logistic weights coincide at $z_0=0$. By Cor.~\ref{cor:glm} the value is at most $\operatorname{rank}(A)=2$, independent of the ambient dimension. The benchmark procedure of Sec.~\ref{subsec:benchmark} realizes it: the nonlinear least-squares solve over the terminal coefficients, using only $f$ and $\nabla f$ along the trajectory and no knowledge of $z^\star$, places $z_2$ at the minimizer to machine precision, run to tolerance $\varepsilon=10^{-10}$ in Fig.~\ref{fig:glm}. The two-query count is the clairvoyant benchmark of Sec.~\ref{subsec:benchmark}, not a runnable two-step method. For scale, gradient descent with step $2/(\mu+L)$ and the triple-momentum method~\cite{VanScoyFreemanLynch2018} with its standard constants, tuned to the restricted curvature, take about $1260$ and $190$ gradient queries to the same tolerance. The per-instance value lies far below the rate the class-optimal methods incur on the very same instance, set by the conditioning rather than the rank, with no contradiction to the accelerated lower bound, which governs the hardest instance of the class~\cite{NemirovskiYudin1983,Nesterov2018}.
\end{example}
\begin{figure}[t]
\centering
\includegraphics[width=0.99\columnwidth]{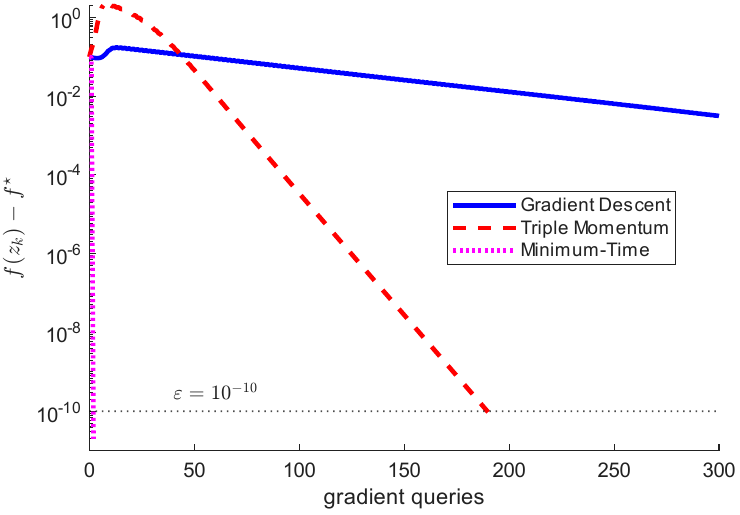}
\caption{Minimum-time schedule on the rank-two logistic instance of~\eqref{eq:log_loss}. 
The schedule, constructed by the benchmark procedure of Sec.~\ref{subsec:benchmark}, reaches the minimizer at the second query, the value $\operatorname{rank}(A)=2$ of Cor.~\ref{cor:glm}. 
Triple momentum and gradient descent take about $190$ and $1260$ queries, respectively.}
\label{fig:glm}
\end{figure}

The quadratic value of Sec.~\ref{subsec:convex} and the bound~\eqref{eq:glm} are one phenomenon: in each the reachable span is confined to a proper subspace forced by the instance, a $\nabla^2 f$-invariant subspace in the quadratic case and the row space of the data here, and $V(z_0)$ is at most the dimension of that subspace rather than scaling with the ambient dimension $n$. The confinement is not a consequence of convexity. The instance $f(z)=\log(1+\|z\|^2)$ is real-analytic with unique critical point $z^\star=0$, yet nonconvex, its radial Hessian eigenvalue $2(1-\|z\|^2)/(1+\|z\|^2)^2$ negative for $\|z\|>1$; its gradient $2z/(1+\|z\|^2)$ is parallel to $z-z^\star$ at every $z$, so every reachable slice lies on the line $\R z_0$ and the single step~\eqref{eq:fo-class} with $c=\inner{z_0}{g_0}/\|g_0\|^2$ places $z_1=z^\star$: $V_0(z_0)=1$ at every start $z_0\notin\T_\varepsilon$, in every dimension, with the realizing coefficient a function of the observed pair alone.

\subsection{The curvature shortcut}
\label{subsec:curvature}

We call the quadratic $\tfrac12(z-z^\star)^\top\nabla^2 f(z^\star)(z-z^\star)$, with the curvature of $f$ at its minimizer, the matched quadratic of $f$, and its controllability index in the sense of Sec.~\ref{sec:necessary} the matched controllability index. On a quadratic the value equals this index. Off the quadratic, curvature can lower it: we give a coplanarity determinant whose simple nonzero roots certify that a $(d-1)$-step slice already meets the target, one step below the index, a reduction the matched quadratic cannot achieve.

Throughout this subsection we write $d$ for the ambient dimension; let $f\in C^3(\R^d)$, $d\ge3$, be strongly convex with minimizer $z^\star$, and let $z_0$ have $g_0:=\nabla f(z_0)\neq0$ and $w_0:=z_0-z^\star$. Let $c$ be the first-step coefficient, $z_1(c)=z_0-c\,g_0$, and for $d\ge4$ fix an admissible continuation, a choice of the remaining $d-3$ steps for which the gradients $g_0,g_1(c),\dots,g_{d-2}(c)$ stay independent on an interval of $c$; for $d=3$ no continuation is needed. Set
\begin{equation}
\Psi(c)=\det\!\big[\,w_0,\ g_0,\ g_1(c),\ \dots,\ g_{d-2}(c)\,\big].
\label{eq:coplanarity}
\end{equation}
Each $g_j(c)$ is $C^2$ in $c$ along the continuation, so $\Psi$ is $C^2$, with roots that depend on the continuation chosen; for $d=3$, $\Psi$ is intrinsic.

\begin{lemma}[Coplanarity criterion]
\label{lem:coplanarity}
A $(d-1)$-step trajectory with first step $c$ and the fixed continuation reaches $z^\star$ if and only if $\Psi(c)=0$; its reachable slice is $(d-1)$-dimensional if and only if $g_0,\dots,g_{d-2}(c)$ are independent.
\end{lemma}

\begin{proof}
The reachable slice is $S(c)=\operatorname{span}\{g_0,g_1(c),\dots,g_{d-2}(c)\}$ and the $(d-1)$-step iterate set is $z_0-S(c)$, which contains $z^\star$ if and only if $w_0\in S(c)$, that is, if and only if the $d$ columns of~\eqref{eq:coplanarity} are linearly dependent, $\Psi(c)=0$.
\end{proof}

Lemma~\ref{lem:coplanarity} certifies reaching in $d-1$ steps; the value can fall further, to a floor set by the least horizon whose trajectory carries enough free step coefficients to place the target in the reachable slice. Define
\begin{equation}
K^\star(d)=\Big\lceil\tfrac{\sqrt{8d+1}-1}{2}\Big\rceil=\min\bigl\{K\ge1:\ \tbinom{K+1}{2}\ge d\bigr\},
\label{eq:kstar}
\end{equation}
the least $K$ whose triangular number $\tbinom{K+1}{2}$ reaches $d$, with $K^\star(d)=K$ on $\tbinom{K}{2}<d\le\tbinom{K+1}{2}$.

\begin{theorem}[Curvature shortcut]
\label{thm:curvature}
Let $f$ be strongly convex on $\R^d$, $d\ge3$, with minimizer $z^\star$, and let $K^\star(d)$ be as in~\eqref{eq:kstar}.
\begin{enumerate}
\item[(a)] If $f\in C^3$ and, at a start $z_0$, the matched-quadratic grade of $w_0$ is $d$ and $\Psi$ in~\eqref{eq:coplanarity} has a simple zero $c^\star\neq0$ at which $g_0,\dots,g_{d-2}(c^\star)$ are independent, then $V_0(z_0)\le d-1$, and $V_0\le d-1$ on a neighborhood of $z_0$; the matched quadratic $\tfrac12(z-z^\star)^\top\nabla^2 f(z^\star)(z-z^\star)$ has $V_0=d$, so the reduction is strictly nonlinear.
\item[(b)] If $f$ is real-analytic, then for $f$ in a residual set of the strongly convex real-analytic objectives, $V_0(z_0)\ge K^\star(d)$ for almost every $z_0$; and there exist such $f$ and a nonempty open set of starts on which $V_0(z_0)=K^\star(d)$.
\item[(c)] $K^\star(d)=d-1$ for $d\in\{3,4\}$, while $K^\star(d)\le d-2$ for $d\ge5$, and $d-K^\star(d)\to\infty$.
\end{enumerate}
\end{theorem}

The proof is deferred to Appendix~\ref{app:thm2}. Part~(a) is a condition on the individual start $z_0$ that can be verified on a given instance; part~(b) is a statement about the typical objective, holding for all $f$ outside a negligible subset of the class and for almost every start. Together they place the per-instance floor at $K^\star(d)$: off a measure-zero set no schedule of~\eqref{eq:fo-class} reaches the target in fewer than $K^\star(d)$ steps, and the floor is met on an open set. By~(c) the one-step reduction of part~(a) is confined to $d\le4$; from $d=5$ onward it compounds, and the gap below the matched controllability index grows without bound, of order $d-\sqrt{2d}$. Figure~\ref{fig:staircase} traces the two counts across dimension: the matched index grows as $d$ while the floor $K^\star(d)$ grows as $\sqrt{2d}$, and the markers, the smallest $V_0$ observed over random real-analytic instances at each $d$, sit on the floor. The lower bound in~(b) is generic, not claimed for every real-analytic strongly convex $f$. As with the quadratic statements, this is the exact-reaching value $V_0$; the tolerance value $V$ agrees with it below the smallest nonzero gradient gap along the trajectory, as after Thm.~\ref{thm:controllability}.

A small value is a statement that the instance carries little reachable structure, not a bound forcing any method to be slow: it certifies curvature as a resource, an instance intrinsically easier to reach than its matched quadratic. The floor is the clairvoyant optimum of Sec.~\ref{subsec:benchmark}, generally not achieved by a runnable schedule; the gap a given method leaves to it is the price of causality.

\begin{figure}[t]
\centering
\includegraphics[width=0.99\columnwidth]{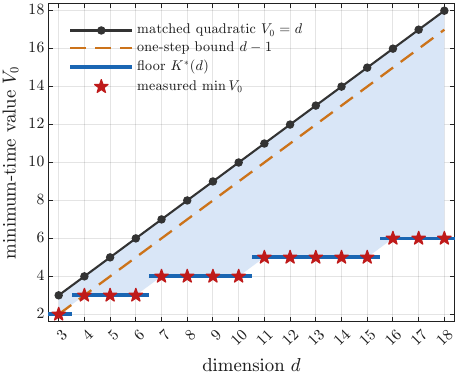}
\caption{Per-instance minimum-time value $V_0$ across dimension $d$. The curvature floor $K^\star(d)$ of~\eqref{eq:kstar} falls strictly below the matched controllability index $d$, the gap growing without bound (Thm.~\ref{thm:curvature}); the markers, the smallest $V_0$ observed over random real-analytic instances at each $d$, coincide with $K^\star(d)$.}
\label{fig:staircase}
\end{figure}

For $d=3$, \eqref{eq:coplanarity} is the triple product $\Phi(c)=\inner{z_0-z^\star}{g_0\times g_1(c)}$ with $g_1(c)=\nabla f(z_0-c\,g_0)$; the matched quadratic gives $\Phi_{\mathrm q}(c)=-c\,\det[w_0,Qw_0,Q^2w_0]$, which vanishes only at $c=0$. Theorem~\ref{thm:curvature} then yields $V_0(z_0)=2<3$ whenever $\Phi$ has a simple nonzero root at which $g_0$ and $g_1$ are independent and $z^\star-z_0\notin\operatorname{span}\{g_0\}$.

\begin{example}[Curvature shortcut in dimension three]
\label{ex:curvature}
Consider
\begin{equation}
\label{eq:probeinstance}
f(z)=\tfrac12 z^\top A z+\gamma\sum_{i=1}^{3}\log\cosh(z_i),
\end{equation}
with $A=\operatorname{diag}(1.38,2.32,4.75)$, $\gamma=\tfrac32$, and $z^\star=0$, real-analytic and $\mu$-strongly convex with $\mu=1.38$. At $z_0=(1,\tfrac1{10},-\tfrac9{10})$ the matched quadratic $\tfrac12 z^\top(A+\gamma I)z$ has grade three, while $\Phi$ has two simple nonzero roots, of which $c^\star\approx-2.44$ has sign opposite to any descent step, at which two coefficients place $z_2=z^\star$; so $V_0(z_0)=2<3$ and the shortcut persists on a neighborhood. Figure~\ref{fig:curvature} plots both determinants against the first-step coefficient: $\Phi$ crosses zero at two simple nonzero roots, each certifying a two-step schedule, while the matched quadratic's $\Phi_{\mathrm q}$ is linear in $c$ and vanishes only at $c=0$. The two greedy first steps, exact-line-search steepest descent and the least-squares step $c=\inner{z_0-z^\star}{g_0}/\|g_0\|^2$, are both positive, so the minimum-time direction is one no descent step probes.
\begin{figure}[t]
\centering
\includegraphics[width=0.99\columnwidth]{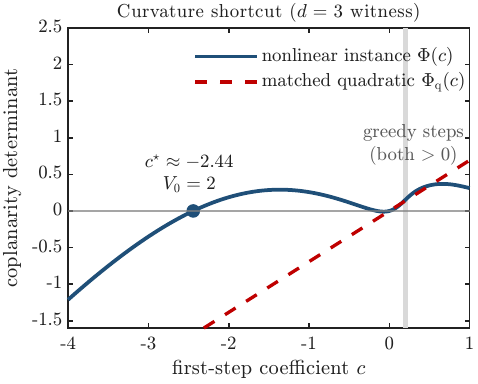}
\caption{Curvature shortcut on $d=3$ for the nonlinear instance~\eqref{eq:probeinstance}: the coplanarity determinant $\Phi(c)$ has two simple nonzero roots, at either of which a two-step schedule reaches $z^\star$, so $V_0(z_0)=2$, one below the matched controllability index, whereas the matched quadratic $\Phi_{\mathrm q}(c)$ vanishes only at $c=0$. The shortcut root $c^\star$ has sign opposite to both greedy first steps, so the minimum-time direction is one no descent step probes; see Example~\ref{ex:curvature} for the values.}
\label{fig:curvature}
\end{figure}
\end{example}

\begin{remark}[Scope]
\label{rem:curvature-scope}
The certificate of part~(a) applies inside any $d$-dimensional reachable structure, for instance the rank-$d$ generalized linear model of Cor.~\ref{cor:glm}, on an open set of starts of positive measure.
\end{remark}

\begin{corollary}[Computable certificate]
\label{cor:certificate}
For $f\in C^2$ and a stationarity-reachable start $z_0$ (Definition~\ref{def:reachable}), the value obeys $V(z_0)\le V_0(z_0)\le\dim\Grad(z_0)$, and $\dim\Grad(z_0)$ is evaluable from first-order data alone as the dimension at which the reachable gradient span~\eqref{eq:span} stabilizes; for the generalized linear model of Cor.~\ref{cor:glm} it equals $\operatorname{rank}(A)$.
\end{corollary}
\begin{proof}
The bound $V(z_0)\le V_0(z_0)$ holds because a critical point lies in $\T_\varepsilon$. For $V_0(z_0)\le\dim\Grad(z_0)$: every iterate lies in $z_0-\Grad(z_0)$, and by Lemma~\ref{lem:spangrowth} the reachable span grows by one dimension at each step until it meets the critical set or freezes; reachability excludes a frozen slice disjoint from the critical set, so a critical point is reached within $\dim\Grad(z_0)$ steps. The span~\eqref{eq:span} is generated by the gradients along trajectories from $z_0$ and stabilizes at $\dim\Grad(z_0)$, hence is obtained from first-order evaluations alone. The generalized-linear-model identity is Cor.~\ref{cor:glm}.
\end{proof}

\begin{remark}[Reach of the certificate]
\label{rem:computable}
A deficiency $\dim\Grad(z_0)<n$ is certified by any nonzero $w$ orthogonal to every reachable gradient; when a proper subspace is $\nabla^2 f$-invariant along the reachable slices, $w$ is a common eigendirection of $\nabla^2 f$ along those slices orthogonal to $g_0$, the first-order reading of the Popov--Belevitch--Hautus test. The value $V_0$ itself is not certified this way from local data: by the curvature shortcut of Thm.~\ref{thm:curvature} it can fall below the grade of $g_0$ with respect to $\nabla^2 f(z_0)$, so no local bound read from the Hessian spectrum at $z_0$ pins $V_0$ from below; the benchmark's clairvoyance is intrinsic, not an artifact of the definition.
\end{remark}

\section{Conclusion}
\label{sec:conclusion}

This paper cast per-instance first-order complexity as a discrete-time
minimum-time optimal control problem: the iterate is the state, the span
coefficients are the control, the gradient-tolerance set is the target, and
the value of the resulting problem is the complexity of the instance.
Everything else follows from this formulation. Its optimality system, obtained by dynamic programming and the discrete Pontryagin principle through a nested formulation, reduces on a strongly convex quadratic to the residual-orthogonality (Galerkin) condition of conjugate gradient, whose iterates emerge as the minimum-time trajectory; the value is a controllability index, and controllability becomes the central invariant, the reachable span of the iteration being the controllable subspace and the per-instance complexity the number of steps needed to fill it. Off the quadratic a reachability test governs feasibility through a Hessian-generated span and bounds the value by the dimension of the reachable structure the instance carries, independent of the ambient dimension. The test also exposes curvature as a resource: a checkable condition drives the per-instance count strictly below the matched controllability index, and in higher dimension the reduction compounds to a floor of order $\sqrt{2d}$ whose gap to the index grows without bound, the quadratic being the boundary case at which the two agree. A minimax
inequality makes the hardest instance's per-instance value a lower bound on the
worst-case complexity of any class, loose precisely where the instance is
structured, of order $\nu$, the number of distinct eigenvalues, against $\sqrt{\kappa}\log(1/\varepsilon)$ on a
clustered quadratic spectrum. Beyond the quadratic the hardest-instance value is bounded by the largest reachable-span dimension the class carries rather than by a contraction rate; whether the gap stays strict on smooth nonconvex classes, by the dimension-versus-rate mechanism seen on the quadratic, is open and tied to nonconvex worst-case lower bounds. The value function is a benchmark for intrinsic
instance difficulty rather than a runnable method, its optimal schedule the
open-loop solution of the inner slice problem, in closed form on a quadratic and, off it, globally solvable but generally not in closed form.

\appendices

\section{Proof of Proposition~\ref{prop:generic}}
\label{app:thm1}
\begin{proof}
The first claim is Thm.~\ref{thm:controllability}(b) with the growth hypothesis holding at every slice, so no continuation freezes and the span reaches $\R^n$ in at most $n$ steps. For the genericity claim, write the reachable span as the stabilized limit of the ascending chain $R_1(z_0)\subseteq R_2(z_0)\subseteq\cdots$ with $R_1=\operatorname{span}\{\nabla f(z_0)\}$ and
\begin{align}
&R_{k+1}(z_0)=R_k(z_0) \notag\\
&+\operatorname{span}\{\nabla^2 f(z)\,u:\ z\in z_0-R_k(z_0),\ u\in R_k(z_0)\},
\label{eq:reachchain}
\end{align}
which equals $\Grad(z_0)$ at stabilization: by Lemma~\ref{lem:spangrowth} the chain stops growing exactly when $R_k$ is $\nabla^2 f$-invariant along its slice. For real-analytic $f$ the map $t\mapsto\nabla^2 f(z_0-tu)\,u$ is real-analytic, so by the identity theorem the span of its values over $t\in\R$ equals the span of its derivatives at $t=0$; the slice in~\eqref{eq:reachchain} may therefore be replaced by the iterated directional derivatives of $\nabla^2 f$ at the single point $z_0$, and, by polarization, each generator is a fixed multilinear form in the partial derivatives of $f$ at $z_0$. Hence for each order bound $N$ the subspace $R^{(N)}(z_0)$ generated using derivatives of order at most $N$ is the column span of a matrix $M_N(z_0)$ with real-analytic entries. Let 
\begin{equation}
    n_\star=\max_{z_0,N}\operatorname{rank}M_N(z_0), 
\end{equation}
an integer at most $n$, achieved at some $(z_0^\circ,N_0)$, and let $\Theta(z_0)$ be the sum of squares of the $n_\star\times n_\star$ minors of $M_{N_0}(z_0)$, real-analytic in $z_0$ with $\Theta(z_0^\circ)\neq0$. Since $\R^n$ is connected, the identity theorem gives that $\{\Theta=0\}$ has measure zero; off it $\operatorname{rank}M_{N_0}(z_0)=n_\star$, so $\dim\Grad(z_0)=n_\star$ there, while $\dim\Grad(z_0)\le n_\star$ everywhere, whence $\dim\Grad(z_0)=n_\star$ outside a set of measure zero. Finally, if $n_\star<n$ then $\Grad(z_0^\circ)$ is a proper subspace, $\nabla^2 f$-invariant along its slice by the stabilization of~\eqref{eq:reachchain}; contrapositively, the absence of such a subspace forces $n_\star=n$.
\end{proof}

\section{Proof of Theorem~\ref{thm:curvature}}
\label{app:thm2}
\begin{proof}
(a) At $c^\star$ the gradients are independent, so $S(c^\star)$ is a $(d-1)$-plane, and $\Psi(c^\star)=0$ gives $w_0\in S(c^\star)$ by Lemma~\ref{lem:coplanarity}: a $(d-1)$-step trajectory reaches $z^\star$, so $V_0(z_0)\le d-1$. Since $f\in C^3$, the gradients along the continuation are $C^2$ in $(z_0,c)$, so $\Psi$ is $C^2$; at $(z_0,c^\star)$ one has $\Psi=0$ and $\partial_c\Psi=\Psi'(c^\star)\neq0$, so by the implicit function theorem $c^\star$ continues to a $C^1$ map $z\mapsto c^\star(z)$ with $\Psi(z,c^\star(z))=0$ on a neighborhood, on which, the gradients remaining independent, $V_0\le d-1$. For the matched quadratic $\nabla f(z)=Q(z-z^\star)$ every gradient lies in $\mathcal K=\operatorname{span}\{Qw_0,\dots,Q^{d-1}w_0\}$, of dimension $d-1$ at grade $d$, with $w_0\notin\mathcal K$; at any $c$ with independent gradients these span $\mathcal K$, so the columns $w_0,g_0,\dots,g_{d-2}(c)$ are independent and $\Psi_{\mathrm q}(c)\neq0$, whence no $(d-1)$-step capture and $V_0=d$.

(b) Set $z^\star=0$. By~\eqref{eq:reachspan} a length-$K$ trajectory reaches the target if and only if $z_0\in S_K(z_0)$, with $g_j=\nabla f(z_j)$. Parametrize such a trajectory by the start $z_0$, the interior coefficients $\theta\in\R^{\binom{K}{2}}$ that fix $z_1,\dots,z_{K-1}$ through~\eqref{eq:fo-class}, and terminal coefficients $c\in\R^K$, and write the reaching condition as
\begin{equation}
\Xi(z_0,\theta,c)=z_0-\sum_{i=0}^{K-1}c_i\,g_i=0,\qquad \Xi:\ \R^{\,d+\binom{K}{2}+K}\to\R^d,
\label{eq:reachmap}
\end{equation}
real-analytic in its arguments; the starts admitting such a trajectory form the image of $\Xi^{-1}(0)$ under $\pi:(z_0,\theta,c)\mapsto z_0$. At any zero with $z_0\neq z^\star$ some terminal coefficient is nonzero, since $z_0=\sum_i c_i g_i$; let $j$ be the largest index with $c_j\neq0$, so that $c_{j+1}=\dots=c_{K-1}=0$ and $\Xi=z_0-\sum_{i\le j}c_i g_i$. For generic interior coefficients the iterates $z_0,\dots,z_{K-1}$ are distinct, and each $z_i$ is fixed through~\eqref{eq:fo-class} by the gradients observed strictly before it. Take $\delta f(z)=\langle v,z-z_j\rangle\,\chi(z)$ with $\chi$ a smooth cutoff equal to $1$ near $z_j$ and supported away from $z_0,\dots,z_{j-1}$; it leaves those earlier iterates, and hence the point $z_j$ itself, in place, adds $v$ to $\nabla f(z_j)$ while leaving $\nabla^2 f(z_j)$ unchanged, and preserves strong convexity for small $\|v\|$. This moves $\Xi$ by $-c_j v$, which ranges over all of $\R^d$, so $(f,z_0,\theta,c)\mapsto\Xi$ is a submersion at every zero, and by parametric transversality~\cite{GuilleminPollack1974} $0$ is a regular value of $\Xi$ for $f$ in a residual set. For such $f$, $\Xi^{-1}(0)$ is a real-analytic manifold of dimension $d+\tbinom{K}{2}+K-d=\tbinom{K+1}{2}$, so $\pi(\Xi^{-1}(0))$ has dimension at most $\tbinom{K+1}{2}$. With $K=K^\star(d)-1$, minimality in~\eqref{eq:kstar} gives $\tbinom{K+1}{2}=\tbinom{K^\star(d)}{2}<d$, so the starts reaching in fewer than $K^\star(d)$ steps, a finite union of such images, have measure zero, and $V_0(z_0)\ge K^\star(d)$ almost everywhere. For attainment, fix $K=K^\star(d)$, for which $\tbinom{K}{2}\ge d-K$ by~\eqref{eq:kstar}; reaching means $\Phi(z_0,\theta):=P^\perp z_0=0$ with $P^\perp$ the orthogonal projection onto $S_K(z_0)^\perp$, which is $d-K$ equations when $S_K(z_0)$ is $K$-dimensional. By the same construction there is such an $f$ and a solution $(z_0^\sharp,\theta^\sharp)$ at which $\partial\Phi/\partial\theta$ has rank $d-K$, and the implicit function theorem then solves $d-K$ coordinates of $\theta$ for a $C^1$ map $z_0\mapsto\theta(z_0)$ with $\Phi(z_0,\theta(z_0))=0$ near $z_0^\sharp$; there $V_0=K^\star(d)$, so the floor is met on an open set. For $d=3$ the lower bound is elementary and holds for every qualifying $f$: since $z^\star-z_0\notin\operatorname{span}\{g_0\}$, $w_0$ lies in no one-dimensional slice, so $V_0=2$.

(c) The triangular numbers satisfy $\tbinom{K+1}{2}\ge d$ first at $K=2$ for $d\in\{2,3\}$ and at $K=3$ for $d\in\{4,5,6\}$, so $K^\star(3)=2$ and $K^\star(4)=3$, matching $d-1$, while $K^\star(5)=K^\star(6)=3\le d-2$. For $d\ge5$ we have $\tbinom{d-1}{2}\ge d$, hence $K^\star(d)\le d-2$; and $K^\star(d)=\Theta(\sqrt d)$ by~\eqref{eq:kstar}, so $d-K^\star(d)\to\infty$.
\end{proof}

\section*{References}
\bibliographystyle{IEEEtran}
\bibliography{refs}




\end{document}